\documentclass[12pt]{amsart}
\usepackage[framemethod=TikZ]{mdframed}

\PassOptionsToPackage{usenames}{xcolor}
\PassOptionsToPackage{dvipsnames}{xcolor}

\usepackage{%
amsmath,			
amssymb,            
amsthm,  			
array,              
graphicx,			
cancel,             
multirow,           
fancyhdr,           
framed,             
mdframed,			
xypic,              
color,              
nicefrac,			
rotating,			
multicol,           
caption,
subcaption,
xcolor,
tikz,
pgfplots
}

\usepackage[%
backend=bibtex,
style=alphabetic,sorting=nyt,
sortcites=true,doi=false,url=false,
firstinits=true,hyperref]{biblatex}
\usetikzlibrary{calc}
\usetikzlibrary{svg.path}
\usetikzlibrary{decorations.markings}
\usepgfplotslibrary{groupplots}

\usepackage{hyperref} 

\definecolor{skyblue}{RGB}{135,206,250}

\newcommand{\set}[1]{\left\{ #1 \right\}}
\newcommand{\Hom}[0]{\operatorname{Hom}}
\newcommand{\mc}[1]{{\mathcal{#1}}}	

\newcommand{\mf}[1]{{\mathfrak{#1}}}
\newcommand{\complex}{{\mathbb C}}
\newcommand{\integ}{{\mathbb Z}}
\newcommand{\real}[0]{\mathbb{R}}
\newcommand{\minv}{^{-1}}
\newcommand{\myspan}[0]{\operatorname*{span}}
\newcommand{\id}[0]{\operatorname*{id}}
\newcommand{\tr}[0]{\operatorname*{tr}}
\newcommand{\pr}[0]{\operatorname{pr}}	
\newcommand{\inv}[1]{#1^{-1}}				
\newcommand{\eq}[2]{\begin{equation} #1 \label{#2} \end{equation} } 
	 	 
\newcommand{\Ad}[0]{\operatorname{Ad}}	 	 
\newcommand{\normm}[1]{{\left\| #1 \right\| }} 
\newcommand{\inn}[2]{\displaystyle \left\langle #1, #2 \right\rangle}

\makeatletter
\newtheorem*{rep@theorem}{\rep@title}
\newcommand{\newreptheorem}[2]{%
\newenvironment{rep#1}[1]{%
 \def\rep@title{#2 \ref{##1}}%
  \begin{rep@theorem}}%
   {\end{rep@theorem}}}
   \makeatother

\newtheorem{definition}{Definition}[section]
\newtheorem{theorem}[definition]{Theorem}
\newreptheorem{theorem}{Theorem}
\newtheorem{proposition}[definition]{Proposition}
\newtheorem{corollary}[definition]{Corollary}
\newtheorem{lemma}[definition]{Lemma}
\newtheorem{remark}[definition]{Remark}

\hypersetup{
    pdffitwindow=false,     
    pdfstartview={FitH},    
    pdfauthor={Tyler Holden},    
    colorlinks=true,        
    linkcolor=blue,         
    citecolor=blue,         
}

\begin{document}


\title{Non-abelian convexity of based loop groups}
\author{Tyler Holden}\email{tholden@math.toronto.edu}

\begin{abstract}
	If $K$ is a compact, connected, simply connected Lie group, its based loop group $\Omega K$ is endowed with a Hamiltonian $S^1 \times T$ action, where $T$ is a maximal torus of $K$. Atiyah and Pressley \cite{Atiyah.Pressley1983} examined the image of $\Omega K$ under the moment map $\mu$, while Jeffrey and Mare \cite{JeffreyMare2010} examined the corresponding image of the real locus $\Omega K^\tau$ for a compatible anti-symplectic involution $\tau$. Both papers generalize well known results in finite dimensions, specifically the Atiyah-Guillemin-Sternberg theorem, and Duistermaat's convexity theorem. In the spirit of Kirwan's convexity theorem \cite{Kirwan1984}, this paper aims to further generalize the two aforementioned results by demonstrating convexity of $\Omega K$ and its real locus $\Omega K^\tau$ in the full non-abelian regime, resulting from the Hamiltonian $S^1\times K$ action. In particular, this is done by appealing to the Bruhat decomposition of the algebraic (affine) Grassmannian, and appealing to the ``highest weight polytope'' results for Borel-invariant varieties of Guillemin and Sjamaar \cite{GuilleminSjamaar2006} and Goldberg \cite{Goldberg2009}.
\end{abstract}

\maketitle


\section{Introduction} 
\label{sec:Introduction}

Let $K$ be a compact group with Lie algebra $\mf k$, and consider the collection of Sobolev class $H^1$ based loops on $K$:
$$ \Omega K = \set{ \gamma \in H^1(S^1,K): \gamma(e_{S^1}) = e_K}.$$
This set has the structure of a Hilbert manifold and inherits a group structure via pointwise multiplication of loops. Elaborated upon in Section \ref{sec:Setup}, $\Omega K$ can be endowed with a K\"ahler structure and a canonical Hamiltonian $S^1\times K$ action with corresponding moment map $\mu: \Omega K \to \real \oplus \mf k^*$.

Generalizing upon the finite dimensional results of the Atiyah-Guillemin-Sternberg Theorem \cite{Atiyah1982, Guillemin.Sternberg1984}, Atiyah and Pressley \cite{Atiyah.Pressley1983} showed that by restricting the $S^1\times K$ action to that of a maximal torus $S^1 \times T$, the image $\mu(\Omega K)$ is an unbounded convex region, generated as the convex hull of infinitely many discrete points. Generalizing the work of Duistermaat \cite{Duistermaat1983}, Jeffrey and Mare \cite{JeffreyMare2010} showed that if $\Omega K$ is additionally endowed with an anti-symplectic involution $\tau: \Omega K \to \Omega K$, then the image of $\Omega K$ under the moment map coincides with that of its real locus $\Omega K^\tau$.  The purpose of this paper is two-fold: The first is to describe the image of $\Omega K$ for the non-abelian moment map, while the second is to give a non-abelian version of Duistermaat convexity. In particular, if $\Omega_\text{alg} K \subseteq \Omega K$ is the subgroup of algebraic maps, and $\Delta(\Omega K) = \mu(\Omega K) \cap \mf t_+^*$ then our two main theorems are as follows:

\begin{reptheorem}{thm:NonabelianConvexity}
	The set $\displaystyle \Delta(\Omega_\text{alg} K)$ is convex.
\end{reptheorem}

\begin{reptheorem}{thm:AlgebraicConvexity}
	If $\Omega_\text{alg} K^\tau$ are the $\tau$-fixed points of $\Omega_\text{alg}K$, then $\displaystyle \Delta(\Omega_\text{alg} K^\tau) = \Delta(\Omega_\text{alg} K).$
\end{reptheorem}

In both cases, the corresponding result for $\Omega K$ is an immediate corollary. Moreover, Theorem \ref{thm:AlgebraicConvexity} will answer in the affirmative an open question speculated upon in \cite{JeffreyMare2010}. 

Our approach mirrors the one utilized in \cite{Atiyah.Pressley1983}, wherein the authors examined the convexity properties of a distinguished dense subset $\Omega_\text{alg} K \subseteq \Omega K$. The set $\Omega_\text{alg} K$ has the fortuitous advantage of admitting a cellular decomposition, for which the closure of each cell is a (possibly singular) finite dimensional projective variety. By embedding this variety into a K\"ahler manifold, the results of \cite{Atiyah1982} apply, yielding convexity of the cell closure. From this, convexity of $\Omega_\text{alg} K$ and consequently $\Omega K$ can be inferred. The approach described in this treatise will also utilize the cellular decomposition of $\Omega_\text{alg} K$, but instead will use the results of \cite{Goldberg2009,GuilleminSjamaar2006, OShea.Sjamaar2000} to demonstrate convexity on each cell closure.


\section{Setup} 
\label{sec:Setup}

There is a great deal of overlap in the foundation material, notation, and preliminaries required to prove our two main results. 
This section is dedicated to establishing that common ground, while each of the subsequent sections will introduce theorem-specific results.

\paragraph{\bf The Based Loop Group:} Let $K$, $T$, and $\Omega K$ be as described in Section \ref{sec:Introduction}. 
Denote by $K_\complex$ and $T_\complex$ the complexification of $K$ and $T$ respectively; by $\mf k, \mf t$ the corresponding Lie algebras; and by $\mf k_\complex, \mf t_\complex$ the complexification of those Lie algebras.
To our choice of $T_\complex$ there is an associated Borel subgroup, denoted by $B\subset K_\complex$.

We can endow $\Omega K$ with the structure of a symplectic (in fact K\"ahler) manifold as follows: The Killing form  on $\mf k$ yields an $\Ad$-invariant inner product $\inn\cdot\cdot_\mf k$, and since $\Omega K$ is a homogeneous space \cite{Pressley.Segal1986} it is enough to give the K\"ahler form at the identity:
$$ \omega_{\Omega K}(X,Y) = \frac1{2\pi} \int_0^{2\pi} \inn{X}{Y'}_\mf k, \qquad X,Y \in C^\infty(S^1,\mf k).$$
There is a natural Hamiltonian $S^1 \times K$-action on $\Omega K$, given by
$$ \left[(s,k) \cdot \gamma\right](\theta) = k \gamma(\theta + s) \gamma(s)\minv k\minv, \qquad s \in S^1, k \in K, \gamma \in \Omega K$$
wherein the $S^1$ action is effectively loop rotation.
Using our $\Ad$-invariant inner product, we identify $\mf k \cong \mf k^*$ and define a moment map $\mu: \Omega K \to \real\oplus \mf k$ by 
\eq{ \mu(\gamma) = \left( \frac1{4\pi} \int_{S^1} \normm{\gamma\minv \gamma'}, \frac1{2\pi} \int_{S^1} \gamma\minv\gamma' \right),}{eq:MomentMap}
where the norm is that induced by $\inn\cdot\cdot_\mf k$.
If $H \subseteq K$ is a closed subgroup of $K$ (most notably when $H$ is a torus), we will denote by $\mu_H$ the corresponding moment map for the $S^1 \times H$ action; that is $\mu_H = (\id\times \pr_\mf h)\circ \mu$ where $\pr_{\mf h}: \mf k \to \mf h$ is orthogonal projection.

\paragraph{\bf Algebraic Loops:} Critical to our study will be the \emph{algebraic based loop group} $\Omega_\text{alg} K$, which arises as the dense subgroup of $\Omega K$ whose maps are the restriction of algebraic maps $\complex^\times \to K_\complex$. 
This group exhibits many of the properties of a finite dimensional reductive algebraic group, for our purposes the most important of which is the Bruhat decomposition. 

Let $L_\text{alg} K_\complex = \set{ f:\complex^\times \to K^\complex: f\text{ algebraic}}$ and $L_\text{alg}^+K_\complex$ be those maps which extend holomorphically to the interior of the unit disk. 
Consider the evaluation map $\text{ev}_0: L_\text{alg}^+ K_\complex \to K_\complex$ given by $\text{ev}_0(\gamma) = \gamma(0)$.
The \emph{Iwahori subgroup} $\mc B$ is the preimage of the Borel under $\text{ev}_0$; that is, $\mc B = \text{ev}_0\minv(B)$.
The Bruhat decomposition is
\eq{\Omega_\text{alg} K = \bigsqcup_{\lambda \in X_*(T)} \mc B \lambda}{eq:BruhatDecomposition}
where each $\mc B\lambda$ is a complex cell indexed by the coweights $X_*(T)$.
The closure of each cell is a finite dimensional projective variety, and has an explicit description as the union of all previous cells under the Bruhat order; namely
\eq{ \overline{\mc B\lambda} = \bigsqcup_{\eta\leq\lambda} \mc B \eta.}{eq:BruhatClosure}
We refer to $\overline{\mc B\lambda}$ as Schubert varieties, in analog of the finite dimensional case. 
We direct the reader to \cite{Mitchell1988} for more information on the Bruhat decomposition of $\Omega_\text{alg} K$.

\paragraph{\bf Filtrations and Embeddings:} The advantage of examining the individual cells $\overline{\mc B\lambda}$ is that each admits an embedding into a finite dimensional Grassmannian, which we describe here.

Choosing a faithful unitary representation realizes $K$ as a subgroup of $SU(n)$, and permits us to view $\Omega_\text{alg} K$ as maps described by finite Fourier series.
For any fixed $m \in \mathbb N$ one can define the space
\eq{ \Omega_{\text{alg},m}K  = \set{ \gamma \in\Omega_\text{alg} K: \gamma(z) = \sum_{k=-m}^m A_k z^k },}{eq:LoopFiltration}
which yields an $S^1 \times K$ equivariant filtration of $\Omega_\text{alg}K$
$$ K = \Omega_{\text{alg},0}K \subseteq \Omega_{\text{alg},1}K  \subseteq \Omega_{\text{alg},2}K  \subseteq \cdots \subseteq \Omega_\text{alg} K.$$
In general these $\Omega_{\text{alg},m}K$ are singular varieties, but come with the advantage that for any coweight $\lambda$ there is a sufficiently large $m$ such that $\overline{\mc B\lambda} \hookrightarrow \Omega_{\text{alg},m} K$ is an embedding of varieties \cite{Atiyah.Pressley1983,Mare2010}. The $\Omega_{\text{alg},m}K$ will act as an intermediary for embedding $\overline{\mc B\lambda}$ into a K\"ahler manifold.


For the moment, let $K$ be a centerfree compact group with Lie algebra $\mf k$.
Define 
$\mc H^{\mf k} = L^2(S^1,\mf k_\complex)$ to be those loops which are square integrable with respect to the Hermitian inner product $\inn xy_{\mf k_\complex}$ induced by the Killing form on $\mf k_\complex$.
The group $L_\text{alg} K_\complex$ acts naturally on $\mc H^{\mf k}$ via the adjoint action
$$ (\gamma\cdot f)(z) = \Ad_{\gamma(z)}f(z), \qquad \gamma \in L_\text{alg} K_\complex, f \in \mc H^{\mf k}.$$
If the complex dimension of $\mf k_\complex$ is $n$, choose a basis $\set{\epsilon_i}_{i=1}^n$ for $\mf k_\complex$ so that
$$\set{ \epsilon_i z^k: i\in\set{1,\ldots,n}, k \in \integ}$$
is a basis for $\mc H^{\mf k}$. 
We polarize $\mc H^{\mf k} = \mc H_+ \oplus \mc H_-$ by taking $\mc H_+ = \myspan_{k\geq 0} \set{\epsilon_iz^k}$ and setting $\mc H_-$ to be its orthogonal complement in the $L^2$ inner product.
The \emph{Grassmannian} of $\mc H^{\mf k}$ is
$$Gr(\mc H^{\mf k}) = \set{ W \subseteq \mc H^{\mf k}: \substack{ \pr_+: W \to \mc H_+ \text{ Fredholm} \\ \pr_-: W \to \mc H_- \text{ Hilbert-Schmidt}}}.$$
This is an infinite dimensional manifold modelled on the space of Hilbert-Schmidt operators $HS(\mc H_+,\mc H_-)$, 
and has a K\"ahler form at $T_{\mc H_+} Gr(\mc H^{\mf k})$ is given by
\eq{ \omega_{\text{HS}}(X,Y) = -i\tr(X^*Y-Y^*X), \qquad X,Y \in HS(\mc H_+,\mc H_-).}{eq:GrassmannSymplectic}

Analogous to \eqref{eq:LoopFiltration} we define the following subspace:
$$Gr_{0,m}(\mc H^{\mf k}) = \set{ W \in Gr(\mc H^{\mf k}): z^m \mc H_+ \subseteq W \subseteq z^{-m} \mc H_+}.$$
As each $Gr_{0,m}(\mc H^{\mf k})$ is homogeneous beyond the $z^{\pm m}$ basis component, one can quickly check that $Gr_{0,m}(\mc H^\mf k)$ is isomorphic to the finite dimensional Grassmannian of complex $nm$-dimensional subspaces of $\complex^{2nm}$ via the identification $W \mapsto W/z^m \mc H_+$. 
For the sake of notation, we will explicitly indicate when we are conceptualizing the finite dimensional Grassmannian by writing $\mc G_m:=Gr_\complex(nm,2nm)$.

The \emph{algebraic Grassmannian} $Gr_0(\mc H^\mf k)$ is the dense subspace of $Gr(\mc H^{\mf k})$ formed by the union 
$$Gr_0(\mc H^{\mf k}):=\bigcup_k Gr_{0,m}(\mc H^{\mf k})$$ 
and as such has a natural $\mathbb N$-indexed filtration given by the $\mc G_m$.
Additionally, we can define a $\complex^\times \times K_\complex$ action on $Gr_0(\mc H^\mf k)$ as follows: Let $W \in Gr_0(\mc H^\mf k)$,
\begin{enumerate}
	\item For $a \in K_\complex, a\cdot W = \set{ \Ad_a f(z) : f \in W}$,
	\item For $s \in \complex^\times, s \cdot W = \set{ f(sz): f \in W}$.
\end{enumerate}
It is easy to check that both are group actions, and in fact we recognize that the $K_\complex$-action is just the $L_\text{alg} K_\complex$ action on $\mc H^\mf k$ restricted to the constant maps.

If $\omega_{FS}'$ is the Fubini-Study form on $\mc G_m$, the following lemma ensures that the various symplectic structures are all mutually compatible.


\begin{lemma}[{\cite[Proposition 2.3]{HaradaHolmJeffreyMare2006}}]
	Let $M = {2nm\choose nm}$ and  $(\mathbb P^{M}, \omega_\text{FS})$ be projective space with the Fubini-Study form.
	If $(Gr_0(\mc H^\mf k), \omega_\text{HS})$ is the algebraic Grassmannian with the Hilbert-Schmidt form, and $p: \mc G_m \hookrightarrow \mathbb P^M$ is the Pl\"ucker embedding, then the restriction of $\omega_\text{HS}$ to $\mc G_m$ is precisely $\omega'_\text{FS} = p^*\omega_\text{FS}$.
\end{lemma}

As alluded to earlier, these Grassmannians will serve as hosts for our embedding. Of particular interest will be the collection of $W \in Gr_0(\mc H^\mf k)$ which satisfy the following three properties:
\begin{enumerate}
	\item $zW \subseteq W$,
	\item $zW = \bar W^\perp$,
	\item If $W_\text{sm}$ consists of the smooth maps, then $W_{\text{sm}}$ is involutive under the Lie bracket on $\mf k$.
\end{enumerate}
We will denote the set of all such spaces as $Gr_0^\mf k$.
Intersecting with the filtration of $Gr_0(\mc H^\mf k)$ yields a filtration of $Gr_0^\mf k$ by the components $Gr_{0,m}^\mf k := \mc G_m \cap Gr_0^\mf k$. 
The following theorem tells us that the $\Omega_{\text{alg},m} K$ are precisely the $Gr_{0,m}^\mf k$ under an appropriate diffeomorphism.

\begin{theorem}\label{thm:PSDiff}
	The action of $L_\text{alg} K_\complex$ on $\mc H^{\mf k}$ extends to an action on $Gr_0(\mc H^{\mf k})$ which preserves $Gr_0^{\mf k}$.
	The map $\phi: \Omega_\text{alg}K \to Gr_0(\mc H^\mf k)$ given by $\gamma \mapsto \gamma \cdot \mc H_+$ defines an $S^1\times K$ equivariant symplectic embedding whose image is $Gr_0^\mf k$, and moreover $\phi(\Omega_{\text{alg},m}K) = Gr_{0,m}^\mf k$, preserving the filtration.
\end{theorem}

\begin{proof}
	This result is a combination of several other results, compactly stated as a single theorem for clarity. 
	The fact that the $L_\text{alg} K_\complex$ action extends to $Gr_0^{\mf k}$, preserves $Gr_0^\mf k$, and that $\phi$ is an embedding whose image is $Gr_0^\mf k$, is a simple adaptation of \cite[Theorem 8.6.2]{Pressley.Segal1986}. 
	
	The proof that $\phi$ is a symplectic embedding is sketched in \cite[Lemma 2.4]{Liu2003}, though there is an error in the result in that paper. 
	In particular, on page 2930, the correct expressions for the evaluation of the symplectic forms should be
	$$ \omega_{\text{AP}}(X,Y) = i\sum_k k \tr(X_k^* Y_k) = \Theta(f_X,f_Y),$$
	where our $\omega_{\Omega K}$ and $\omega_{\text{HS}}$ are the $\omega_\text{AP}$ and $\Theta$ of the paper.
	In light of this error, we give a full proof of the result in the appendix [Lemma \ref{lemma:SymplecticEmbedding}].
	
	Equivariance follows quickly once we recognize that if $a \in A_\complex$ then $\Ad_a \mc H_+ = \mc H_+$. 
	As $(a\cdot\gamma)(z) = a \gamma(z) \inv a$ one has
	\begin{align*}
		\phi(a\cdot \gamma) &= \Ad_{a\gamma\inv a}\mc H_+ = \Ad_a\circ \Ad_\gamma \circ\Ad_{\inv a} \mc H_+ \\
		&= \Ad_a \circ \Ad_\gamma \mc H_ + \\
		&= a \cdot \phi(\gamma). \\ 
	\end{align*}
	For $s \in S^1$, $(s\cdot\gamma)(z) = \gamma(sz)\gamma(s)\minv$. The map $f(z) \mapsto f(\inv s z)$ is bijective on $\mc H_+$, so that
	$$ 	\phi(s \cdot \gamma) = \Ad_{\gamma(sz)\gamma(\inv s)} \mc H_+  = \Ad_{\gamma(sz)} \mc H_+ = s \cdot \phi(\gamma).$$
	Both actions commute, so $\phi$ is equivariant as required. Finally, the $\complex^\times\times K_\complex$-action commutes with $z$-multiplication, leaving $z^m \mc H_+$ invariant and preserving the filtration.
\end{proof}

\begin{remark}
	The triviality of the centre of $K$ is essential only in the proof of \cite[Theorem 8.6.2]{Pressley.Segal1986}.
	In the case when $K$ is not centerfree, $\Ad(K) = K/Z(K)$ is a centerfree group and $K \to \Ad(K)$ is a finite covering.
	Hence $\Omega_\text{alg}K$ is a union of connected components of $\Omega_\text{alg}\!\Ad(K)$.
\end{remark}

\paragraph{\bf The Determinant Bundle:}
If $k<n$ and $Gr(k,n)$ represents the Grassmannian of complex $k$-planes in $\complex^n$, then $Gr(k,n)$ has a {tautological} rank $k$ vector bundle, $\gamma_{k,n} \to Gr(k,n)$ where
$$ \gamma_{k,n} = \set{ (W,w): W \in Gr(k,n), w \in W}.$$
Since this is a rank $k$ vector bundle, we may define a line bundle $\det^* \to Gr(k,n)$ by taking the $k$-th exterior power over each fibre.
More explicitly, if $\set{w_1,\ldots,w_k}$ is a basis for $W$, a typical fibre element over $W$ will be of the form $(W, \alpha w_1 \wedge \cdots \wedge w_k)$, which we write as $(W,[\alpha,w])$.
Naturally, choosing another basis should not change the structure of the fibre.
If $\set{w_1',\ldots,w_k'}$ is another basis and $C$ is the change of basis matrix from $w$ to $w'$, then 
$$w_1' \wedge \cdots \wedge w_k' = (\det C) w_1 \wedge \cdots \wedge w_k,$$ 
so we identify the elements $(W,[\alpha,w']) \sim (W,[\alpha(\det C), w])$.

We are more interested in the dual bundle $\det\to Gr(k,n)$. It is easy to check that $\det \to Gr(k,n)$ is the pullback of $\mc O(1) \to \mathbb P^{n\choose k}$ by the Pl\"ucker embedding $p: Gr(k,n) \to \mathbb P^{n\choose k}$. If $\nabla_{\mc O(1)}$ is the Chern connection on $\mc O(1)$, then $\nabla = p^*\nabla_{\mc O(1)}$ is a Chern connection on $\det$, and moreover the curvature satisfies
$$ F_{p^*\nabla_{\mc O(1)}} = p^* F_{\nabla_{\mc O(1)}} = -2i\pi p^* \omega_{FS} = -2i\pi \omega_{FS}',$$
which shows that $(\det,\nabla)$ is a positive prequantum line bundle. When applied to our $\mc G_m$, we shall denote the bundle by $\det_m \to \mc G_m$.

Visually, the following diagram summarizes the content of this section:

\centerline{\xymatrix{%
	& & \det_m \ar[d] \ar[r] & \mc O(1) \ar[d] \\
	\overline{\mc B\lambda} \ar@{^{(}->}[r] & \Omega_{\text{alg},m}K \ar@{^{(}->}[r] & \mc G_m \ar@{^{(}->}[r] & \mathbb P^{2nm \choose nm}
}}


\section{Non-abelian convexity} 
\label{sec:Non-abelian convexity}

Herein we examine the image of the moment map $\mu: \Omega K \to \real \oplus \mf k$ for the $S^1\times K$-action. It is well known that when the group is non-abelian, the image of the moment is not necessarily convex. Rather, Kirwan \cite{Kirwan1984} demonstrated that the correct analog lies in considering that portion of the moment map image which dwells within the positive Weyl chamber $\mf t^*_+$.  Several authors have generalized the aforementioned work, though our particular interest lies with the ``highest weight polytope'' approach of Brion \cite{Brion1987}, who showed that $K_\complex$-invariant subvarieties have convex image in $\mf t_+^*$. Guillemin and Sjamaar \cite{GuilleminSjamaar2006} extended Brion's work to those varieties invariant under just the action of the Borel subgroup $B$, and it is this result that we exploit here. The following is the main result of that paper, paraphrased to omit the structure of the highest weight polytope whose particular structure is not necessary for our result.

\begin{theorem}[{\cite[Theorem 2.1]{GuilleminSjamaar2006}}]\label{thm:GS}
	Let $M$ be a compact, complex manifold and $L \to M$ a positive Hermitian line bundle with Hermitian connection $\nabla$ and curvature form $\omega$. Let $K$ be a compact Lie group with complexification $K_\complex$ and $B\subseteq K_\complex$ a Borel subgroup. Assume that $K$ acts on $L$ by line bundle automorphisms, which preserve the complex structure on $M$ and Hermitian structure on $L$. If $X$ is a $B$-invariant irreducible closed analytic subvariety of $M$, then $\Delta (X) = \mu(X) \cap \mf t_+^*$ is a convex polytope, where $\mu: M \to \mf k^*$ is the moment map corresponding to $\omega$. 
\end{theorem}

Let $\mu: \Omega K \to \mf g \oplus \real$ be the moment map in \eqref{eq:MomentMap} and define $\Delta(X) = \mu(X) \cap \mf t^*_+$. By applying Theorem \ref{thm:GS} to the Bruhat cells $\mc B\lambda$ of $\Omega_\text{alg} K$, we deduce the following theorem:

\begin{theorem}\label{thm:NonabelianConvexity}
	$\Delta(\Omega_\text{alg}K)$ is convex.
\end{theorem}

\begin{proof}
	Fix some $\lambda \in X_*(T)$ and consider $X=\overline{\mc B\lambda} \subseteq \Omega_\text{alg} K$. Since the Iwahori subgroup $\mc B$ is the preimage of the Borel subgroup under the evaluation map, $\overline{\mc B\lambda}$ is clearly $B$-invariant. In Section \ref{sec:Setup} we showed that for sufficiently large $m \in \mathbb N$, there is an embedding $\overline{\mc B\lambda} \hookrightarrow \Omega_{\text{alg},m} K \hookrightarrow \mc G_m$, and a positive prequantum line bundle $\det_m \to \mc G_m$, whose curvature form is the K\"ahler form $\omega_{FS}'$. 
	
To define a $\complex^\times \times K_\complex$ linearization on $\det_m$, we will first define one on $\det_m^*$. 
Let $W \in \mc G_m$ and choose a homogeneous basis $\set{w_i}$ for $W$. Elements of $\det_m^*$ look like $(W,[\alpha, w])$ where $w=w_1\wedge\cdots \wedge w_n$ and $\alpha \in \complex$. 
Let $s \in \complex^\times$ and let $s$ act on each $w_i$ by loop rotation, so that $s \cdot w_i = s^{k_i} w_i$, for some $k_i \in \integ$. This action is diagonal in this basis, and the induced action of $\complex^\times$ on $\Lambda^{\text{top}} W$ is given by
$$ s \cdot (W,[\alpha, w_1\wedge\cdots\wedge w_n]) = (s\cdot W, [\alpha, w_1(sz) \wedge \cdots \wedge w_n(sz)]) = (s\cdot W, [\alpha s^k, w]),$$
where $k=k_1+\cdots+k_n$.
Similarly, if $k \in K_\complex$ then $k \cdot w_i = kw_i$ is certainly linear. Moreover, since $K_\complex$ is connected and semisimple, $\det(k) =1$:
$$ k\cdot (W,[\alpha, w]) = (k\cdot W, [\alpha, (kw_1)\wedge\cdots \wedge (kw_n)]) = (k\cdot W, [\alpha, w]).$$
These actions commute and hence define a $\complex^\times \times K_\complex$ action on $\det_m^*$ which commutes with the projection map. 

The $K_\complex$ action trivially preserves the Hermitian structure on $\det_m$. The $\complex^\times$ action is also Hermitian, since
$$ \inn{s \cdot[\alpha_1, w]}{s\cdot [\alpha_2,w]} = \inn{[s^k\alpha_1,w]}{[s^k\alpha_2,w]} = s^k s^{-k} \inn{[\alpha_1,w]}{[\alpha_2,w]} = \inn{[\alpha_1,w]}{[\alpha_2,w]} .$$

By passing to the dual, one derives an action of $\complex^\times \times K_\complex$ on $\det_m$. Since the Hermitian structure on $\det_m$ is just the dual of that on $\det_m^*$, one also concludes that the Hermitian structure on $\det_m$ is $\complex^\times \times K_\complex$-invariant. Furthermore, the $\complex^\times \times K_\complex$-action acts holomorphically on $\mc G_m$. Since $S^1 \times K$ includes into $\complex^\times \times K_\complex$ as its compact real form, we conclude that the corresponding real group action is Hermitian and preserves the complex structure. Theorem \ref{thm:GS} thus implies that $\Delta(\overline{\mc B\lambda})$ is convex.

To extend from each Schubert variety to the whole algebraic based loop group, choose any two points $x_1,x_2 \in \Delta(\Omega_\text{alg}K)$ and take $p_i \in \Omega_\text{alg} K$ such that $\mu(p_i) = x_i$. There exist $\lambda_1,\lambda_2 \in X_*(T)$ such that $p_i \in \mc B\lambda_i$, and since $X_*(T)$ is directed, a $\lambda$ such that both $\mc B\lambda_1 \subseteq \mc B\lambda$ and $\mc B\lambda_2 \subseteq \mc B\lambda$. Consequently, both $p_1,p_2 \in \overline{\mc B\lambda}$ whose image is convex, as required.
\end{proof}

\begin{corollary}\label{cor:H1Convexity}
	The image $\Delta(\Omega K)$ is also convex.
\end{corollary}

\begin{proof}
	The algebraic based loops are dense in $\Omega K$ \cite{Atiyah.Pressley1983}. Since the moment map is continuous, $\mu(\Omega_\text{alg} K)$ is dense in $\mu(\Omega K)$, and so $\overline{\Delta(\Omega_\text{alg} K)} = \Delta(\Omega K)$. Since the closure of a convex set is convex, the result follows.
\end{proof}


\section{Duistermaat-type Convexity} 
\label{sec:Duistermaat-type Convexity}

We now turn our attention to the real locus of the based loop group; a discussion which necessitates introducing an anti-symplectic involution on $\Omega K$. Assume that $K$ is equipped with an involutive group automorphism $\sigma: K \to K$ such that if $T \subseteq K$ is a maximal torus, then $\sigma(t) = \inv t$ for every $t \in T$. Such an involution is guaranteed to exist by \cite{LoosBook}. The differential $d_e\sigma: \mf k \to \mf k$ is thus involutive as well and induces the eigenspace decomposition $\mf k = \mf p \oplus \mf q$, where $\mf p$ and $\mf q$ are the $\pm1$-eigenspaces respectively. Extend $d_e\sigma$ to $\mf k_\complex$ in an anti-holomorphic fashion by setting
\eq{ \hat \sigma: \mf k_\complex \to \mf k_\complex, \qquad X+iY \mapsto d_e\sigma(X) -i d_e\sigma(Y).}{eq:AntiholomorphicExtension}
We define an involutive group automorphism $\tau: \Omega K \to \Omega K$ by $(\tau \gamma)(z) = \sigma(\gamma(\bar z))$ for $\gamma \in \Omega K$.
One can easily check that $\tau$ leaves $\Omega_\text{alg} K$ invariant, and so also defines an involution there. It is straightforward to check that $\tau$ also preserves the filtration $\Omega_{\text{alg},m} K$. We will often conflate $\tau$ with the corresponding $\integ_2$ action it induces on $\Omega K$.

The $\tau$ fixed points of $\Omega K$ are denoted $\Omega K^\tau$, and are often referred to as the \emph{real locus} of $\Omega K$. The nomenclature is derivative of the finite dimensional regime, where our manifold is a complex variety and $\tau$ is complex conjugation. In \cite{JeffreyMare2010}, Jeffrey and Mare showed that $\mu_T(\Omega K) = \mu_T(\Omega K^\tau)$; an analog of Duistermaat's convexity theorem in finite dimensions \cite{Duistermaat1983}. In the former paper, the question of whether $\mu_T(\Omega_\text{alg} K) =\mu_T(\Omega_\text{alg} K^\tau)$ was proposed. Our main theorem in this section will answer the more general non-abelian case in the affirmative, and give the result for the full class of Sobolev $H^1$-loops as an immediate corollary.

Once again, our strategy will be to demonstrate the result on the Schubert varieties $\overline{\mc B\lambda}$. Guillemin and Sjamaar \cite{GuilleminSjamaar2006} generalize Duistermaat's convexity result in finite dimensions to the non-abelian regime for singular varieties. Goldberg \cite{Goldberg2009} in turn combined this result with that of \cite{OShea.Sjamaar2000} to derive the same result for singular $B$-invariant varieties.

\begin{definition}[{\cite{GuilleminSjamaar2006}}]
	Let $U$ be a compact group with an involutive automorphism $\tilde\sigma: U \to U$. A \emph{Hamiltonian $(U,\tilde\sigma)$-manifold} is a quadruple $(M,\omega,\tau,\mu)$ such that $(M,\omega)$ is a smooth symplectic manifold, endowed with a Hamiltonian $U$-action for which $\mu: M \to \mf u^*$ is the moment map. In addition, $\tau$ is an anti-symplectic involution, compatible with the $U$-action as follows:
\begin{enumerate}
	\item $\mu(\tau(m)) = -(d_e\tilde\sigma)^*(\mu(m))$,
	\item $\tau(u\cdot m) = \tilde\sigma(u) \tau(m)$.
\end{enumerate}
	A \emph{$(U,\tilde\sigma)$-pair} $(X,Y) \subseteq M \times M^\tau$ is a pair such that 
	\begin{enumerate}
		\item $X$ is a $U$- and $\tau$-stable irreducible closed complex subvariety of $M$,
		\item $X_\text{reg} \cap M^\tau \neq \emptyset$, where $X_\text{reg}$ are the regular points of $X$, 
		\item $Y$ is the closure of a connected component of $X_\text{reg} \cap M^\tau$.
	\end{enumerate}
\end{definition}

\begin{theorem}[{\cite[Theorem 1.7]{Goldberg2009}}]\label{thm:Goldberg}
	Let $(M,\omega, \tau, \mu)$ be a compact, connected, $(U,\tilde\sigma)$-manifold, equipped with a complex structure compatible with $\omega$, relative to which $\tau$ is anti-holomorphic. Assume that the Borel subgroup $B \subseteq K_\complex$ is preserved by the anti-holomorphic extension of $\tilde \sigma$, and take $\mf u = \mf p \oplus \mf q$ to be the $\pm 1$-eigenspace decomposition of $\mf u$ with respect to $d_e\tilde\sigma$. 
Let $L \to M$ be a holomorphic line bundle with connection $\nabla$ such that the curvature form $F_\nabla = -2\pi i \omega$,  and assume $\tau$ lifts to an involutive, anti-holomorphic bundle map $\tau_L$ which preserves $\nabla$.
If $(X,Y)$ is a $(U,\tilde\sigma)$-pair such that $X$ is $B$-invariant, $\tau$-invariant, and $X^\tau$ is non-empty, then 
\eq{\Delta (Y) = \Delta(X) \cap \mf m^*}{eq:NonAbelianConvexity}
where $\mf m^* \subseteq \mf q^*$ is a maximal abelian subspace of $\mf q^*$.
\end{theorem}

The Schubert varieties $\overline{\mc B\lambda}$ will play the role of $X$, while the Grassmannian $\mc G_m$ will play that of $M$. To begin, we convince ourselves that $\overline{\mc B\lambda}$ is a candidate for $X$, by showing that it is preserved by $\tau$ and that the embedding $\overline{\mc B\lambda} \hookrightarrow \mc G_m$ is $\integ_2$-equivariant. The following lemma is partially sketched in \cite[Theorem 5.9]{Mitchell1988}, and was communicated to us by Mare.

\begin{proposition}\label{prop:BruhatTau}
	Each Schubert variety $\overline{\mc B\lambda}$ is invariant under $\tau$. 
\end{proposition}

\begin{proof}
	We first claim that the Borel subgroup $B$ is invariant under the complexification of $\sigma$. Write $\mf k_\complex$ in its root decomposition as
	$$ \mf k_\complex = \mf t_\complex \oplus \bigoplus_{\alpha \in \Phi} \mf k^\complex_\alpha$$
	where $\mf t_\complex$ is the complexification of $\mf t$ and $\mf k^\complex_\alpha$ are the corresponding root spaces. 
	Let $\hat \sigma: \mf k_\complex \to \mf k_\complex$ be the anti-holomorphic map defined in \eqref{eq:AntiholomorphicExtension}.
	It is sufficient to show that $\mf b$ is invariant under $\hat \sigma$. 
	Our choice of Borel corresponds to 
	$$ \mf b = \mf t_\complex \oplus \bigoplus_{\alpha \in \Phi_+} k^\complex_\alpha,$$
	and the definition of $\sigma$ ensures that $\mf t_\complex$ is invariant under $\hat\sigma$, so we need only show that each $\mf k_\alpha^\complex$ is invariant under $\hat\sigma$.
	For this, let $X \in \mf k_\alpha^\complex$ so that $[H,X] = \alpha(H) X$ for all $H \in \mf t_\complex$. Note that it is in fact sufficient to just take $H \in \mf t$, and that $\alpha|_{\mf t} \in i\real$. Since $\hat \sigma$ is a Lie algebra automorphism, this yields
	\begin{align*}
		\hat\sigma[H,X] &= [\hat\sigma(H),\hat\sigma(X)] = -[H,\hat\sigma(X)] \\
		&= \hat\sigma(\alpha(H) X) = \overline{\alpha(H)} \hat\sigma(X) = - \alpha(H) \hat\sigma(X)
	\end{align*}
	so that $[H,\hat\sigma(X)] = \alpha(H) \hat\sigma(X)$ showing that $\mf k_\alpha^\complex$ is invariant under $\hat \sigma$.

	Invariance of $\mc B$ under $\tau$ follows a similar procedure, wherein we demonstrate that $\text{Lie}(\mc B)$ is invariant under $d\tau$.
	We have that
	$$ \text{Lie}(\mc B) = \set{\gamma(z)=\sum_{k=0}^m A_k z^k: A_0 \in \mf b, m \in \mathbb N},$$
	so that if $\gamma \in \text{Lie}(\mc B)$ then
	\begin{align*}
		d\tau(\gamma) &= d\tau \left( \sum_{k=0}^m A_k z^k \right) = \sum_{k=0}^\infty \hat\sigma(A_k \bar z^k) = \sum_{k=0}^\infty \hat\sigma(A_k) z^k.
	\end{align*}
	Clearly $\hat\sigma(A_0) \in \mf b$ by our previous argument, showing that $\mc B$ is invariant under $\tau$. It is easy to check that each group morphism $\lambda:S^1\to T$ is fixed by $\tau$, and hence $\mc B\lambda$ is $\tau$ invariant. 

	Finally, the closure relation on the Schubert varieties is given by
	$$\overline{\mc B\lambda} = \bigsqcup_{\eta \leq \lambda} \mc B\eta,$$
	and as each $\mc B\eta$ is invariant under $\tau$, so too is $\overline{\mc B\lambda}$.
\end{proof}

In light of the previous proposition, it makes sense to write $\overline{\mc B\lambda}^\tau$ and to consider its image under the moment map. We next extend $\tau$ to $Gr_0(\mc H^\mf k)$ in order to utilize the Grassmannian model of the algebraic based loops. Define $\hat\tau: \mc H^\mf k \to \mc H^\mf k$ by $(\hat\tau f)(z) = \hat\sigma(f(\bar z))$, which extends to a map 
$$\hat\tau: Gr_0(\mc H^\mf k) \to Gr_0(\mc H^\mf k), \qquad \hat\tau W = \set{ \hat\sigma(f(\bar z)): f \in W }.$$
The map $\hat \tau$ is well defined on $Gr_0(\mc H^\mf k)$ since $\hat \sigma$ acts by Lie algebra automorphism, and is involutive since $\hat \sigma$ is the derivative of an involutive map.

\begin{proposition}
	The action of $\tau$ on $Gr_0(\mc H^\mf k)$ preserves $Gr_0^\mf k$ and its corresponding filtration. Moreover, if $\phi: \Omega_\text{alg} K \to Gr_0(\mc H^k)$ is the symplectic embedding given in Theorem \ref{thm:PSDiff}, then $\phi$ is $\integ_2$-equivariant.
\end{proposition}

\begin{proof}
We begin by showing that $Gr_0^\mf k$ is preserved. Multiplication by $z$ commutes with $\hat \tau$, since for any $W \in Gr_0(\mc H^\mf k)$ we have
$$ \hat \tau(zW) = \set{ \hat\sigma(\bar z f(\bar z)) : f \in W} = \set{ z \hat\sigma(f(\bar z)): f \in W} = z\cdot \hat \tau W,$$
keeping in mind the anti-holomorphic nature of $\hat\sigma$. Consequently, $\tau$ leaves $z^m\mc H_+$ invariant, preserving the filtration. Furthermore, $\hat \tau$ is an isometry of the Killing form on $\mf k_\complex$ and so commutes with the map $W \mapsto \overline W^\perp$. 
Now if $W \in Gr_0^\mf k$ then $zW \subseteq W$ implies that
$$ z(\hat\tau W) = \hat \tau(zW) \subseteq \hat\tau W,$$
while $zW = \overline W^\perp$ implies that
$$ z(\hat \tau W) = \hat\tau (zW) = \hat\tau(\overline W^\perp) = (\overline{\hat\tau W})^\perp.$$
The involutivity of $W_\text{sm}$ follows immediately from the fact that $\hat\sigma$ is a Lie algebra automorphism, so $Gr_0^\mf k$ is preserved by $\tau$.

All that remains to be shown is the compatibility of $\tau$ with $\hat\tau$ through the embedding $\phi$. Certainly $\hat\tau \mc H_+ = \mc H_+$, and for any $k \in K$ one has $\hat \sigma \circ \Ad_k = \Ad_{\sigma(k)} \circ \hat \sigma$, so that
\begin{align*}
	\hat\tau \phi(\gamma) &= \set{ \hat\sigma ( \Ad_{\gamma(\bar z)} f(\bar z)):f \in \mc H_+} \\
	&= \set{\Ad_{\sigma(\gamma(\bar z))} \hat \sigma(f(\bar z)): f \in \mc H_+} \\
	&= \set{\Ad_{\tau \gamma} f: f \in\mc H_+ } \\
	&= \phi(\tau\gamma).
\end{align*}
\end{proof}

Fix a Schubert variety $\overline{\mc B\lambda}$ and choose $m$ sufficiently large so that $\overline{\mc B\lambda}$ embeds into $\mc G_m$.
Define $\tilde \sigma: S^1 \times K \to S^1 \times K$ by $\tilde \sigma(s,k) = (\inv s, \sigma(k))$. Whenever $(s,t) \in S^1 \times T$ one has $\tilde\sigma(s,t) = (s^{-1}, t^{-1})$. This implies that $\real \oplus \mf t$ is contained in the $(-1)$-eigenspace of $d_e\tilde\sigma$,  $\real \oplus \mf q$, and hence $\mf m = \real \oplus \mf t$. In the context of this particular choice of $\tilde \sigma$, Equation \eqref{eq:NonAbelianConvexity} simply reads $\Delta(Y) = \Delta(X)$.

\begin{lemma}\label{lemma:GmManifold}
	$(\mc G_m, \omega'_\text{FS}, \tau, \mu)$ is a compact, connected, $(S^1 \times K, \tilde \sigma)$-manifold.
\end{lemma}

\begin{proof}
	The majority of the desired properties are immediately true: $S^1\times K$ is a compact group with involution $\tilde \sigma$, $(\mc G_m,\omega_\text{FS})$ is a symplectic manifold with moment map $\mu$ and anti-symplectic involution $\tau$.
Thus all that needs to be checked are the compatibility conditions.
We begin by showing condition (2); namely, if $(s,k) \in S^1 \times K$ then $\tau((s,k)\cdot W) = (\inv s, \sigma(k))\cdot \tau(W)$. 

We will check each action separately. Let $W \in Gr_0(\mc H^\mf k)$, $s \in S^1$, and $k \in K$, so that
\begin{align*}
	\tau(s \cdot W) &= \set{ \hat\sigma\left[ f(s\bar z) \right] : f \in W} = \set{ \hat \sigma \left[ f(\overline{\inv s z}) \right] : f \in W } \\
	&= \inv s \set{ \hat \sigma [f(\bar z)] : f \in W} \\
	&= \inv s \cdot \tau(W).\\
	\tau(k \cdot W) &= \set{ \hat\sigma\left[\Ad_k f(\bar z) \right] : f \in W} = \set{\Ad_{\sigma(k)} \hat\sigma[f(\bar z)]: f \in W} \\
	&= \sigma(k) \cdot \set{ \hat\sigma[f(\bar z)]: f \in W} \\
	&= \sigma(k)\cdot\tau(W).
\end{align*}
On the other hand, since $S^1 \times K$ is connected, \cite[Lemma 2.2]{OShea.Sjamaar2000} implies that condition (1) is true up to an appropriate shifting of $\mu$.
\end{proof}

We are well acquainted with the fact that $\mc G_m$ is in fact K\"ahler, so it certainly has a complex structure which is compatible with $\omega_{FS}'$.
To see that $\tau$ is anti-holomorphic with respect to this structure, we introduce the following result:

\begin{lemma}\label{lemma:AHiffAS}
	If $(M,\omega,J,g)$ is a K\"ahler manifold and $\tau: M \to M$ is an involutive isometry, then $\tau$ is antiholomorphic if and only if $\tau$ is antisymplectic.
\end{lemma}

\begin{proof}
	Assume first that $\tau$ is anti-holomorphic.
In particular, $d\tau \circ J = -J \circ d\tau$.
To see that $\omega$ is anti-symplectic, we have
\begin{align*}
	(\tau^*\omega)(X,Y) &= \omega(d\tau X, d\tau Y) = g(J d\tau X, d\tau Y) \\
	&= -g(d\tau JX, d\tau Y)\\
	&= -g(JX,Y)\\ 
	&= -\omega(X,Y).
\end{align*}

Conversely, assume that $\tau$ is anti-symplectic.
Note that since $\tau$ is involutive it is necessarily bijective, so that $d\tau$ is an isomorphism.
Proceeding in a similar fashion as that above, we have
$$(\tau^* g)(X,Y) = g(d\tau X, d\tau Y) = \omega(d\tau X, J d\tau Y)$$
Since $\tau$ is an isometry, $\tau^* g = g$, so we also have
$$ (\tau^* g)(X,Y) = g(X,Y) = \omega(X,JY) = -\omega (d\tau X, d\tau JY) $$
where in the last equality we have used the fact that $\tau$ is anti-symplectic.
By multilinearity, we thus have
$$ 0 = \omega(d\tau X, Jd\tau Y + d\tau JY).$$
Since $X$ was arbitrary and $d\tau$ is surjective, non-degeneracy of $\omega$ immediately implies that $(J \circ d\tau + d\tau \circ J)Y = 0$, and since $Y$ was arbitrary, $\tau$ is anti-holomorphic.

Note that we did not need involutivity to show that anti-holomorphic implied anti-symplectic.
Similarly, we only needed that $\tau$ was surjective in order to show the converse direction.
\end{proof}

\begin{lemma}\label{lemma:AntiHolLift}
	The involution $\tau: \mc G_m \to \mc G_m$ lifts to an involutive anti-holomorphic bundle map $\tau_{\det_m}: \det_m \to \det_m$.
Moreover, this map preserves the connection; that is, $\tau_{\det_m}^*\nabla = \nabla.$
\end{lemma}

\begin{proof}
Define $\tau_{\det_m}: \det_m \to \det_m$ as follows: Fix $W \in \mc G_m$ and a basis $\set{w_1,\ldots,w_{nm}}$ of $W$, so that the fibre over $W$ is given by $(W,[\alpha,w])$.
Fibrewise, we define
	$$ \tau_{\det_m}: (\text{det}_m)_{W} \to (\text{det}_m)_{\tau W}, \qquad \left( W,[\alpha,w]\right) \mapsto \left( \tau W, [\overline\alpha, \tau(w)]\right).$$
	Certainly if $w$ is a basis then $\tau(w)$ is a basis, convincing us that $\tau w_1 \wedge \cdots \wedge \tau w_k$ is indeed non-zero line element.
	
	To see that the map is well defined, let $w$ and $w'$ be two different bases and $C$ be the change of variable matrix sending $w$ to $w'$; that is, $w'_i = C_{ij} w_j$.
Since $\tau$ is anti-holomorphic
	$$ \tau w_i' = \tau\left[C_{ij} w_j\right] = \overline{C_{ij}} \tau w_j$$
	implying that the change of basis between $\tau w$ and $\tau w'$ is given by $\overline{C_{ij}}$.
We identify $[\alpha, w']$ and $[\alpha\cdot\det C, w]$, so
	$$ \tau\left[ \alpha w_1'\wedge \cdots \wedge w_k'\right] = \overline{\alpha} (\tau w_1') \wedge \cdots \wedge (\tau w_k') = \overline{\alpha} \det{\overline C} (\tau w_1 \wedge \cdots \wedge \tau w_k).$$
	On the other hand, we have
	$$ \tau\left[ \alpha w_1 \wedge \cdots \wedge w_k \right] = \overline{\alpha \det C} (\tau w_1) \wedge \cdots \wedge (\tau w_k),$$
	and these are quite naturally equal since $\overline{\det C} = \det{\overline C}$.

	The map $\tau_{\det_m}$ is clearly involutive, and it is anti-holomorphic since in any trivializing neighbourhood $U$, $\tau: U \times \complex \to U \times \complex$ is separately anti-holomorphic on $U$ and conjugate-linear on $\complex$.
By Hartogs' theorem, $\tau$ is thus anti-holomorphic.

	All that remains to be shown is that $\tau$ preserves the connection.
Since $\tau$ is an involution, it preserves the real part of the hermitian metric $h$.
Since it is additionally anti-holomorphic, by Lemma 4.4 of \cite{OShea.Sjamaar2000} we have that $\tau$ preserves the connection.
\end{proof}

\begin{proposition}\label{prop:CellConvexity}
	If $\lambda \in X_*(T)_+$ is a dominant coweight, then
	$$\Delta(\overline{\mc B\lambda}) = \Delta(\mc B\lambda^\tau).$$
\end{proposition}

\begin{proof}
Choose $m$ sufficiently large so that $\overline{\mc B\lambda}$ embeds into $\mc G_m$. 
We have already seen that $\overline{\mc B\lambda}$ is a closed irreducible subvariety of $\mc G_m$ invariant under $B$. 
Lemmas \ref{lemma:GmManifold} through \ref{lemma:AntiHolLift} show that $\mc G_m$ satisfies the hypotheses of Theorem \ref{thm:Goldberg},  and Lemma \ref{prop:BruhatTau} shows that $\overline{\mc B\lambda}$ is $\tau$-stable.
Furthermore, $\lambda \in \overline{\mc B\lambda}$ is a smooth point of $X$ which is also in $\mc G_m^\tau$, showing that $X_\text{sm} \cap M^\tau \neq\emptyset$. 

Applying Theorem \ref{thm:Goldberg}, if $Y$ is the closure of any connected component of $X_\text{sm} \cap M^\tau$ we will have $\Delta(X) = \Delta(\overline{\mc B\lambda}) = \Delta(Y)$.
Our challenge is thus to exploit this fact to yield $\Delta(\overline{\mc B\lambda}^\tau)$ on the right-hand-side.

Since $\lambda$ is dominant, the smooth locus of $\overline{\mc B\lambda}$ is precisely $\mc B\lambda$ \cite{MalkinEtAl2005}, giving $X_{\text{reg}} \cap M^\tau = \mc B\lambda^\tau$.
This set is closed and consists of only finitely many connected components, say $\mc B\lambda^\tau = \bigsqcup_k C_k$ where each $C_k$ is also closed in $\mc G_m$.
By Theorem \ref{thm:Goldberg} we have $\Delta(\overline{\mc B\lambda}) = \Delta(C_k)$ and hence
$$ \Delta(\mc B\lambda^\tau) = \Delta\left(\sqcup_k C_k\right) = \sqcup_k\Delta(C_k) = \sqcup_k \Delta\left(\overline{\mc B\lambda}\right) = \Delta\left(\overline{\mc B\lambda}\right).$$
\end{proof}

\begin{theorem}\label{thm:AlgebraicConvexity}
	$\Delta(\Omega_\text{alg} K) = \Delta(\Omega_\text{alg}K^\tau).$
\end{theorem}

\begin{proof}
	We can write $\Omega_\text{alg} K$ as a disjoint union of the cells $\mc B\lambda$ as in \eqref{eq:BruhatDecomposition}, so that 
	$$ \Omega_\text{alg} K^\tau = \bigsqcup_{\lambda \in X_*(T)} \mc B\lambda^\tau.$$
	
	Let $c \in \mu_A(\Omega_\text{alg}K)$ and choose any $\gamma \in \Omega_\text{alg}K$ such that $\mu_A(\gamma) = c$.
There exists a unique $\mc B\lambda'$ such that $\gamma \in \mc B\lambda'$, and since the collection of dominant coweights is cofinal in the Bruhat order, there exists a $\lambda$ such that $\gamma \in \mc B\lambda' \subseteq \overline{\mc B\lambda}$.
By Proposition \ref{prop:CellConvexity},
	$$c = \Delta(\gamma) \in \Delta(\overline{\mc B\lambda}) = \Delta(\mc B\lambda^\tau) \subseteq \Delta(\Omega_\text{alg}K^\tau),$$
	showing that $\Delta(\Omega_\text{alg} K) \subseteq \Delta(\Omega_\text{alg} K^\tau)$.
The other inclusion is trivial, and the result then follows.
\end{proof}

One can apply Theorem \ref{thm:AlgebraicConvexity} to immediately deduce several useful corollaries, the first of which is that we can weaken the regularity conditions on loops:

\begin{corollary}\label{cor:TorusMomentAlg}
	Let $A \subseteq K$ be any torus, and $\mu_A = (\id \times \pr_\mf a) \circ \mu: \Omega K \to \real \oplus \mf a$ be the corresponding moment map for the $S^1 \times A$ action, where $\pr_\mf a: \mf k \to \mf a$ is the projection map. Then
	$$ \mu_A(\Omega_\text{alg} K) = \mu_A(\Omega_\text{alg} K^\tau).$$
\end{corollary}

\begin{proof}
	Consider the case where $A = T$ is a maximal torus of $K$. It is known that the image of $\mu_T$ may be derived from that of $\Delta$ by examining the union of the convex hulls of the Weyl orbits of elements in the image of $\Delta$ \cite[Theorem 1.2.2]{Guillemin.SjamaarBook2005}. Applying this to Theorem \ref{thm:AlgebraicConvexity} the result follows for $\mu_T$. If $A$ is any other torus, it is contained in a maximal torus $T$, and as projections of convex set are still convex, the result follows in general.
\end{proof}

\begin{corollary}\label{cor:TorusMomentH1}
	If $A \subseteq K$ is any torus, then
	$$\mu_A(\Omega K) = \mu_A(\Omega K^\tau).$$
\end{corollary}

This result was stated and proved in \cite{JeffreyMare2010}, wherein the authors showed that $\mu_T(\Omega K^\tau)$ is convexity using a convexity theorem of Terng \cite{Terng1993}. Additionally, the fixed points of the $S^1\times T$ action on $\Omega K$ are the group morphisms $\Hom(S^1,T)$, and it is straightforward to see that these are also fixed under the $\tau$ action, yielding the desired result. However, once one has the result for the algebraic based loops, the full class of Sobolev $H^1$ loops quickly follows:

\begin{proof}
	The inclusion $\mu_A(\Omega K^\tau) \subseteq \mu_A(\Omega K)$ is trivial. On the other hand, by \cite{Atiyah.Pressley1983} we know that $\Delta(\Omega K) = \Delta(\Omega_\text{alg}K)$, so by Corollary \ref{cor:TorusMomentAlg}
	$$\mu_A(\Omega K) =\mu_A(\Omega_\text{alg} K^\tau) \subseteq \mu_A(\Omega K^\tau).$$ 
\end{proof}

We hypothesize that the corresponding non-abelian analog of Corollary \ref{cor:TorusMomentAlg} is also true; that is, $\Delta(\Omega K) = \Delta(\Omega K^\tau)$. This would follow immediately if it could be shown that $\Delta(\Omega_\text{alg} K)$ is closed, for then the proof of Corollary \ref{cor:H1Convexity} would imply that $\Delta(\Omega K) = \Delta(\Omega_\text{alg} K)$, and one would only need to replace every instance of $\mu_A$ with $\Delta$ in the proof of Corollary \ref{cor:TorusMomentAlg}. Without knowing that $\Delta(\Omega_\text{alg} K)$ is closed, one can trace through the suggested proof to quickly arrive at the following:

\begin{corollary}
$ \overline{\Delta(\Omega K)} = \overline{\Delta(\Omega K^\tau)}.$
\end{corollary}


\appendix
\section{} 

\begin{lemma}\label{lemma:SymplecticEmbedding}
	Under the embedding $\phi: \Omega_\text{alg} K \hookrightarrow Gr_0(\mc H)$, we have $\phi^* \omega_{HS} = \omega_{\Omega K}$; that is, the symplectic structure on $\Omega_\text{alg} K$ is compatible with the Hilbert-Schmidt symplectic structure under this embedding.
\end{lemma}

\begin{proof}
	We shall proceed by first examining the result in $SU(n)$, then argue the general case afterwards. 
	Let $X,Y \in T_e \Omega_{\text{alg}}SU(n)$, written as a Fourier series
	$$ X(\theta) = \sum_k A_k e^{ik\theta}, \qquad Y(\theta) = \sum_\ell B_\ell e^{i\ell\theta}.$$
	Applying the symplectic form, we get
	\begin{align*}
		\omega_{\Omega K}(X,Y) &= \frac1{2\pi} \int_0^{2\pi} \sum_{k,\ell} \inn{A_k e^{ik\theta}}{i\ell B_\ell e^{i\ell\theta}} d\theta \\
		&= \frac{i}{2\pi} \sum_{k,\ell} \ell \tr(A_k^* B_\ell) \int_0^{2\pi} e^{i(k-\ell)\theta} d\theta \\
		&= i \sum_k k \tr(A_k^*B_k).
	\end{align*}
	On the other hand, the coordinate chart about $W \in Gr_0(\mc H)$ is defined as
	$$U_W = \set{ V \oplus TV: T: W \to W^\perp \text{ Hilbert-Schmidt }}$$
	which allows us to uniquely identify $V$ with the map $T$.
 	If $W = \mc H_+$ then $U_{\mc H_+}$ equivalently consists of those spaces $V$ such that $\pi_V: V \to \mc H_+$ is an isomorphism, in which case $T$ can be explicitly defined as
	$$ T = \pr_-|_W \circ \pi_W\minv: \mc H_+ \to W \to \mc H_-. $$
	In coordinates around the identity, $\phi(\gamma) = \pr_- \circ \inv\pi_{\gamma\mc H_+}$, and so $d\phi_e(X(z)) =\pr_- \circ L_{X(z)},$ where $L_{X(z)}$ is left-multiplication by $X(z)$.

	Let $f_X = d\phi_e(X) = \pr_- \circ L_X$ and $f_Y = d\phi_e(Y) = \pr_- \circ L_Y$, so that 
	\begin{align*}
		\omega_{HS}(f_X,f_Y) &= -i \tr\left( f_X^* f_Y - f_Y^* f_X\right)\\ 
		&= -i \sum_{\substack{i \in \set{1,\ldots,n} \\ j\geq0}} \inn{\epsilon_i z^j}{(f_X^*f_Y - f_Y^*f_X) \epsilon_i z^j}_{L^2} \\
		&= -i \sum_{\substack{i \in \set{1,\ldots,n} \\ j\geq0}} \left[ \inn{f_X \epsilon_i z^j}{f_Y \epsilon_i z^j} - \inn{f_Y \epsilon_i z^j}{f_X \epsilon_i z^j}_{L^2} \right].
	\end{align*}
	Examining just half of this equation for now, we see that
	\begin{align*}
		\sum_{j\geq0}\sum_{i=1}^n\inn{f_X \epsilon_i z^j}{f_Y\epsilon_i z^j}_{L^2} &= \sum_{j\geq0}\sum_{i=1}^n\sum_{r,s<-j}\inn{(A_r\epsilon_i) z^{j+r}}{(B_s\epsilon_i) z^{j+s} }_{L^2} \\
		&= \sum_{j\geq0}\sum_{i=1}^n\sum_{r< -j} \inn {A_r \epsilon_i}{B_r\epsilon_i}_{\mf k_\complex} \\
		&= \sum_{j\geq 0}\sum_{r < -j} \tr(A_r B_r^*) = \sum_{k<0} k \tr(A_kB_k^*)
	\end{align*}
	On the other hand, we know that $A_{-k} = \overline{A_k}$ and hence 
	$$\sum_{k <0}k \tr(A_k B_k^*) = -\sum_{k>0} k \tr(\overline{A_{k}} \overline{B_{k}}^*) = - \sum_{k>0} k \tr( A_k^* B_k) $$
	where in the last inequality we have used transpose- and cyclic-invariance of the trace.
 	By symmetry, we thus have
	$$ \omega_{HS}(f_X,f_Y) = i \sum_{k\geq 0} \left[ k \tr(A_k^* B_k) - k \tr( A_{-k}^* B_{-k} ) \right] = i \sum_k k \tr(A_k^* B_k)$$
	which is the same result we got from $\omega(X,Y)$, as required.

	For a general compact group, assume once again that the group is centerfree. 
	Since $\inn\cdot\cdot_{\mf k_\complex}$ is the inner product induced by the Killing form on $\mf k$ (itself related to the Killing form on $\mf k_\complex$ by conjugating the second argument), we have the $\inn\cdot\cdot_{\mf k_\complex}$ is $\Ad$-invariant. 
	The action of $\Omega_\text{alg} K$ on $Gr_0(\mc H^\mf k)$ thus realizes $\Omega_\text{alg} K$ as a submanifold of $\Omega_{\text{alg}} SU(\mf k_\complex)$, and since $K$ is centerfree this is actually an embedding. 
	Choosing an appropriate basis for $\mf k_\complex$, we identify $SU(\mf k_\complex)$ with $SU(n)$ and $\mc H^{\mf k}$ with $\complex^n$. 
	The result then follows from commutativity of the embeddings:
	\centerline{\xymatrix{ 
		\Omega_\text{alg}SU(n) \ar@{^{(}->}[r] & Gr_0(\mc H^{\complex^n}) \\
		\Omega_\text{alg}K \ar@{^{(}->}[u] \ar@{^{(}->}[r] & Gr_0(\mc H^\mf k)\ar@{^{(}->}[u]
	   }}
	Once again, we are not concerned with the centerfree requirement, since when $K$ is not centerfree we have that $\Omega_\text{alg} K$ is a disjoint union of the connected components of $\Omega \Ad(K)$.
\end{proof}


\printbibliography

\end{document}